\documentclass[a4paper,twoside,12pt]{article}      % Comments after  % are ignored
\usepackage{amsmath,amssymb,amsfonts} % Typical maths resource packages
\usepackage{amsthm}
\usepackage{graphics}                 % Packages to allow inclusion of graphics
\usepackage{color}                    % For creating coloured text and background
\usepackage{hyperref}                 % For creating hyperlinks in cross references
\usepackage{graphicx,eucal}
\usepackage{latexsym,epsfig,epic,eepic}

\newtheorem{theorem}{Theorem}[section]
\newtheorem{proposition}[theorem]{Proposition}

\newtheorem{lemma}[theorem]{Lemma}

\theoremstyle{definition}

\date{2010}
\title{The group of almost-periodic homeomorphisms of the real line}

\author{Bertrand Deroin}

\begin{document}
\maketitle
\begin{abstract}
We study the group of almost-periodic homeomorphisms of the real line. Our main result states that an action of a finitely generated group on the real line without global fixed point is conjugated to an almost-periodic action without almost fixed point. This is equivalent to saying that the action on the real line can be compactified to an action on a $1$-dimensional lamination of a compact space, without global fixed point. As an application we give an alternative proof of Witte's theorem: an amenable left orderable group is locally indicable. 
\end{abstract}

\section{Introduction} 

A group is \textit{left-orderable} if it admits a total order which is invariant by left multiplications. A well-known fact is that a countable left-orderable group acts faithfully on the real line by preserving orientation homeomorphisms~\cite{Ghys}. This dynamical view on left orderable groups has been powerful, see~\cite{Navas}, and it is an interesting problem to look for some structure on the real line invariant by such a group. %If one aim to cover the full generality of left orderable groups, even restricting to the class of finitely generated ones, 
These structures must be "soft", since, for instance, some left-orderable group cannot act by diffeomorphisms of class $C^1$, see~\cite{Navas2}. 

In this note we study \textit{almost-periodic} actions on the real line. The notion of almost-periodicity goes back to Bohr, and has been developed in various directions, especially recently with the development of almost-periodic tilings. An action of a group on the real line is almost-periodic if the real line can be compactified as a leaf in a laminated compact space, in such a way that the action extends to an action by homeomorphisms preserving each individual leaf. 

It is easy to prove that an action of a finitely generated group on the real line is conjugated to an almost periodic action. For instance, by taking the conjugation increasing fastly at infinity. However, such an almost periodic action has an \textit{almost fixed point}: namely, points close to infinity are displaced by the generators a short distance. Our main result is that every action of a finitely generated group on the real line is in fact conjugated to an almost-periodic action without almost-fixed points, thus leading to an action of the group on a $1$-dimensional laminated compact space \textit{without} fixed point. %These actions are obtained by finding a conjugation to a Lipschitz action, such that the displacement of every element of the group is bounded, i.e. $\sup_{x\in {\bf R}} |h(x) - x | < \infty$. 

This new material permits to give an alternative proof of the following theorem by Dave Witte~\cite{W}, which had been conjectured by Linnell~\cite{Linnell}: a left-orderable group which is amenable has a non trivial morphism to the integers. However, we believe that this is interesting in its own, and will serve for other applications.

\section{Almost-periodic representations}\label{s: almost periodic}

The set of continuous functions on the real line, $C^0 (\mathbb R)$, is equipped with the topology of uniform convergence on compact subsets, i.e. the compact open topology. A continuous function $f\in C^0 (\mathbb R)$ is \textit{almost-periodic} if the set of its translates $T_s f (x) = f(x+s)$ is relatively compact in $C^0 (\mathbb R)$.  

The group of homeomorphisms of the real line preserving the orientation, $\mathrm{Homeo}^+ (\mathbb R)$, is also equipped with the compact open topology, which turns it into a topological group. A homeomorphism $ h\in \mathrm{Homeo}^+ (\mathbb R) $  such that $f(x) = h(x) - x$ is an almost periodic function will be called almost-periodic. The set of almost-periodic and orientation preserving homeomorphisms will be denoted $APH^+ ({\mathbb R})$. 

\begin{proposition} 
$APH ^+ (\mathbb R)$  is a subgroup of $\mathrm{Homeo}^+ (\mathbb R)$.
\end{proposition}

\begin{proof}
This is a consequence of the following characterization of almost-periodicity: a homeomorphism from the real line to itself is almost-periodic if the set consisting of its conjugates by all the translations is relatively compact. Indeed, if $h$ belongs to $\mathrm{Homeo}^+(\mathbb R)$ and $\tau_s$ is the translation $\tau_s (x) = x+s$, then $ T_s ( h-\mathrm{id} ) = \tau_{s}^{-1} \circ  h\circ \tau_s - \mathrm{id}$. \end{proof}

An action of a group on the line whose image is contained in $APH^+(\mathbb R)$ will be called almost-periodic. There are various ways to construct a faithful almost-periodic action of a left orderable and countable group $G$ on the line. The most simple is to begin with a faithful action by homeomorphims on the interval, and to extend it to the line by conjugating by the powers of the translation $x\mapsto x+1$. 
%Another way consists in conjugating a given orientation preserving action on the real line by a homeomorphism $x\mapsto \Phi(x)$ such that $\Phi(x)$ tends very slowly to infinity, but it is not clear for us if this procedure works for every countable subgroup of $\mathrm{Homeo}^+ (\mathbb R)$. 
Hence, $APH^+(\mathbb R)$ contains a copy of any left orderable and countable group. 

The typical example of an almost periodic function is a trigonometric polynomial of the form $f ( x ) = \sum _i  a_i \cos( b_i x + c_i)$, where $a_i, b_i, c_i$ are real numbers. Such a function can be seen as the restriction of an analytic function defined on the $n$-dimensional torus to the orbit of an irrational linear flow. This fact turns out to be a characterization of almost-periodicity: a function $f$ is almost-periodic if and only if there exists a compact space $X$ and a continuous flow $\{ \Phi_s \}_{s \in \mathbb R}$ acting on $X$, such that there exists a point $p\in X$ such that $f(x) = F( \Phi_x p)$ for some continuous function $F: X \rightarrow \mathbb R$. The quadruple $(X,\Phi,p,F)$ may be chosen in the following way: $X$ is the closure of the set $\{  T_s f:\ s\in \mathbb R \}$ in $C^0 (\mathbb R)$ for the compact open topology (which by definition is compact), the flow $\Phi_s$ is the operator $T_s$, the point $p$ is the function $f$, and the function $F$ is the evaluation at the point $0$. 

Let us mimic this construction for almost-periodic actions of a given abstract group $G$. Introduce the set $APR(G)$ whose elements are the representations of $G$ in $APH^+ (\mathbb R)$. Then $APR(G)$ can be seen as a closed subset of $APH^+ (\mathbb R)^{\mathcal G}$, which provide a topology on $APR(G)$. Define the \textit{translation flow} $\{\Phi_s\}_{s\in \mathbb R}$ acting on $APR (G) $ by conjugation by the translations $\tau_s (x)= x+s$, namely $$\Phi_s (\rho) (g) : = \tau_s \circ \rho(g) \circ \tau_s^{-1}$$ for every $\rho \in APR(G) $ and $g\in G$. This is a topological flow acting on $APR(G)$. 

Suppose for a moment that the orbit of an element $\rho \in APR (G)$ by the translation flow is the real line, i.e. $\rho$ does not commute with any non trivial translation. A point $s$ of the real line can be put in correspondance with the representation $\Phi_s \rho$, and one may think to the action of the group $G$ on the real line given by the representation $\rho$ as acting on the orbit of $\rho$ by the translation flow. This permits to define a representation $\mathrm{Univ}: G\rightarrow \mathrm{Homeo} (APR(G))$, that will be referred to the \textit{universal} representation:
\[  \mathrm{Univ} (g) (\rho) : = \tau_{-\rho(g) (0)} \circ \rho \circ \tau_{\rho(g)(0)} .\]
The verification that this formula gives rise to a genuine representation is tedious, but straighforward. We leave it to the reader. 

By construction, the action of $G$ on $APR(G)$ defined by $\mathrm{Univ}$ preserves each orbit $\Phi_{\mathbb R} (\rho)$, and is semi-conjugated to $\rho$ on it. More precisely, we have: 
\[ \mathrm{Univ} (g) ( \Phi_{-s} (\rho) ) = \Phi _{- \rho (g) (s)} (\rho) , \]
for every $s\in \mathbb R$, $\rho \in APR(G)$, and $g\in G$. This is the reason for the terminology "universal". 

\begin{proposition} \label{p:closure under translation flow} 
Let $G$ be a finitely generated group and $\rho_0 \in APR(G)$ be an almost-periodic representation. Then the set $X(\rho_0) = \overline{\Phi_{\mathbb R}(\rho_0)}$ is compact, and invariant by the translation flow $\{\Phi_s\}_{s\in \mathbb R}$. Moreover, there is a continuous action of $G$ on $X$, preserving every $\Phi$-orbit, such that for every $\rho \in X(\rho_0)$ which does not commute with any non trivial translation, the action of $G$ on $\Phi_{\mathbb R} (\rho)$ is topologically conjugated to $\rho$. 
\end{proposition} 

Observe that the flow $\Phi$ may have fixed points on $X(\rho_0)$. This phenomena occurs when the representation $\rho_0$ almost commutes with a translation at infinity. However, one can slightly change the construction, so that this does not happen, for instance adding to the group $G$ an almost periodic group which is far from being periodic, including when we look at the action close to infinity.

%there exists 
%\begin{enumerate} 
%\item a compact space $X$, 
%\item a flow $\{\Phi_s\}_{s\in \mathbb R}$ acting continuously on $X$, 
%\item an action of $G$ on $X$, $\rho': G \rightarrow \mathrm{Homeo} (X)$, which preserves each $\Phi$-orbit, and 
%\item a point $p\in X$, 
%\end{enumerate}
%such that for every $s\in \mathbb R$, and every element $g\in G$, we have the following formula: 
%\[ \rho'(g) ( \Phi_{-s} (p) ) = \Phi _{- \rho (g) (s)} (p) . \]
%In other words, $\rho$ is conjugated to the action of $\overline{\rho}$ on the $\Phi$-orbit of $p$. 
%\end{proposition} 

%\begin{proof} Let $\mathcal G$ be a finite system of generators of $G$.  The following formula defines an action of $G$ on $APH(G)$: $\rho'(g) (\rho) := \Phi _{- %\rho(g)(0)} \rho$. We have, for every $s\in \mathbb R$, every $\rho \in APR(G)$, and every $g\in G$, 
%\[ \rho'(g) (\Phi_{-s} (\rho ) ) = \Phi _{-\rho(g) (s) } (\rho).   \]
%The verification is tedious but straighforward; we leave it to the reader. Hence the action of $G$ defined previously preserves each $\Phi$-orbit. Now let us take for $X$ the closure of the orbit of $\rho$ by the flow $\Phi$, which is clearly $\Phi$-invariant, and hence $\rho'$-invariant. Because $\rho$ is almost-periodic, $X$ is compact. The choice $p = \rho$ finishes the proof of the Proposition. 
%\end{proof}

\section{Actions without almost fixed points} 

Let $G$ be a finitely generated group with finite generating set $\mathcal G$ and $\rho \in APR (G)$ an almost-periodic representation. We say that $\rho$ has an \textit{almost fixed point} if 
$$\inf _{x\in \mathbb R} \sup _{g\in \mathcal G} |\rho(g)(x) - x | =0.$$ 
An equivalent way to think about this property is to consider the closure $X(\rho)$ of the orbit of the representation $\rho$ under the translation flow $\Phi$ (see Proposition~\ref{p:closure under translation flow}): then an almost fixed point will provide a global fixed point for the representation $\mathrm{Univ}$ of $G$ on the space $X$. It is not easy to construct almost-periodic actions without almost fixed points. The main new material of this note is to provide such a construction, if the group is finitely generated and left-orderable:

\begin{theorem} 
\label{t: almost-periodic action}
An action of a finitely generated group on the real line is topologically conjugated to a almost-periodic action. Moreover, if the original action has no fixed point, then it is possible to find a conjugacy to a almost-periodic action without almost fixed point. 
\end{theorem}

Observe that it is only necessary to prove the second part of the result, since we can add to the group $G$ a non trivial translation, with the effect of killing every fixed point. 

The proof of Theorem~\ref{t: almost-periodic action} will be done in two steps. First, we will prove that the action is conjugated to an action by  bilipschitz homeomorphisms, and then we will modify the action so that the displacement of every element is uniformly bounded, but the minimal displacements on the left and on the right over the generating set are uniformly bounded from below by a positive constant. Such a representation is almost-periodic and has no almost fixed point, so this will complete the proof.

We denote by $\mathrm{Bilip}^+ (\mathbb R) $ the group of orientation preserving bilipschitz homeomorphisms of the real line. For every $h\in \mathrm{Bilip}^+ (\mathbb R)$, we denote $K(h)$ the minimum of the numbers $K\geq 1$ such that  
\begin{equation} \label{eq: lipschitz}\forall x, y \in{\mathbb R} \ \ \ \ \ \ \ \   K^{-1}\cdot |y-x| \leq |h(y) -h(x) | \leq K \cdot |y-x|.\end{equation} 
We equipp $\mathrm{Bilip}^+ (\mathbb R)$ with the topology of uniform convergence on compact subsets of the real line.

\begin{lemma} \label{l: lipschitz conjugation}
A finitely generated group of homeomorphisms of the real line is conjugated to a group acting by Lipschitz homeomorphisms.  
\end{lemma} 

\begin{proof} Our proof is inspired by a discussion with Marie-Claude Arnaud. In~\cite{DKNP}, we give a more conceptual proof (but more elaborate) based on probabilistic arguments. 

Let $\lambda = f(x) dx$ be a probability measure on ${\mathbb R}$ with a smooth and positive density $f$ such that for $|x|$ big enough, we have $f(x) = 1/x^2$. The following observation will be central in what follows: if, for some constant $L \geq 1$, a homeomorphism $h$ from the real line to itself satisfies
\begin{equation}\label{eq: cauchy condition} (h^{-1}) _* \lambda \leq L \lambda ,\end{equation}
then $h$ is Lipschitz. To prove this fact, first observe that $\lambda([x, +\infty) ) = \frac{1}{x}$, for $x$ a large positive number (and similarly $\lambda (-\infty, x] ) = \frac{1}{|x|}$ if $x$ is large negative number). Thus, the inequality (\ref{eq: cauchy condition}) shows that for $|x|$ large enough, $\frac{1}{| h(x) |} \leq \frac{L}{|x|}$. The density of $(h^{-1}) _* \lambda$ is given by $h'(x) f( h(x) ) $, hence (\ref{eq: cauchy condition}) gives the bound $h' (x) \leq \frac{ L f(x) }{ f (h(x)) }$ for almost every $x$. Thus, up to sets of Lebesgue measure $0$, $h'$ is bounded on every compact interval, and for $|x|$ large enough we have $h'(x) \leq L^3$; this proves that $h'$ is bounded, and hence $h$ is Lipschitz.  

Denote by $G$ a finitely generated subgroup of $\mathrm{Homeo}^+ ({\mathbb R})$, and let $\mathcal G$ be a finite system of generators for $G$. Let $\varphi \in L^1(G)$ be a function with positive values such that, for every element $h\in G$, there is a constant $L_h$ such that $\varphi(hg ) \leq L_h  \varphi(g)$; for instance one can take $\varphi (g) = \alpha ^{||g||}$ with $\alpha$ a small enough positive number, where $||g||$ is the minimum length of a word in the elements of $\mathcal G$ which equals $g$. Normalize the function $\varphi$ so that $\sum_{g\in G} \varphi(g) = 1$, and introduce the probability measure on ${\mathbb R}$ defined by
\[  \nu : = \sum_{g\in G} \varphi (g)\cdot  g_* \lambda .\]
Observe that for every $h\in G$, we have 
\[  h_* \nu = \sum _{g\in G} \varphi (g)\cdot  (hg)_* \lambda \leq L \nu ,\]
where $L= L_{h^{-1}}$. 

The measure $\nu$ has full support and no atoms. Thus, there exists a homeomorphism $\phi$ from the real line to itself which maps $\nu $ to $\lambda$. Denote $h^{\Phi} = \Phi \circ h \circ \Phi^{-1}$. We have 
\[  h^{\Phi} _* \lambda = \Phi_* h_* \nu \leq L \Phi_* \nu = L \lambda . \]
From the discussion above, we deduce that $G^{\Phi}$ is contained in $\mathrm{Bilip}^+(\mathbb R)$.  \end{proof}

Let $\mathcal G$ be a finite symmetric system of generators of $G$ and let $K>1$ and $0< C< D$ some constants. We denote by $R = R (G, \mathcal G, K, C, D)$ the set of representations $\rho : G\rightarrow \mathrm{Bilip}^+(\mathbb R)$ such that for every $g\in \mathcal G$, $K(g) \leq K$ and for every $x\in \mathbb R$: 
\begin{equation}\label{eq: displacement condition}  x-D \leq \min _{g\in \mathcal G} g(x) \leq x - C \leq x+ C \leq \max _{g\in \mathcal G} g(x) \leq x + D \end{equation}
The set $R$ can be seen as a closed subset of $\mathrm{Bilip}^+ (\mathbb R)^G$, and as such is equipped with the product topology; the relations (\ref{eq: lipschitz}) and (\ref{eq: displacement condition}) implies that $R$ is a compact set.

\begin{lemma} 
There are constants $K>1$ and  $C, D>0$ and a finite generating set $\mathcal G$ of $G$ such that $R$ is non empty. 
\end{lemma}

\begin{proof} Let $K$ be a constant such that for every $g\in \mathcal G$, $K(g) \leq K$. The condition (\ref{eq: displacement condition}) might not be satisfied, e.g. when the action is affine. So we will have to modify our action and build a new one. To do so, we define a sequence of points $x_n\in \mathbb R$ for every $n\in {\mathbb Z}$ by $x_0= 0$ and $x_{n+1} = \max _{g\in \mathcal G} g(x_n)$, or equivalently $x_{n-1} = \min _{g\in \mathcal G} g (x_n)$ since $\mathcal G$ is symmetric. Because $G$ has no fixed point on the real line, we have 
\[ \lim _{ n\rightarrow \pm \infty} x_n = \pm \infty. \]  
We let $\varphi$ be the homeomorphism from the real line to itself which sends $x_n$ to $n$, and is affine on the intervals $[x_n, x_{n+1}]$. We claim that the action of $G$ on the real line defined by $\rho (g) = \varphi \circ g \circ \varphi ^{-1}$ belongs to $R(G,\overline{\mathcal G}, K^6,1,4)$ for the generating set $\overline{\mathcal G} = \mathcal G \cup \mathcal G^2$. 

To prove this, we remark that the distortion of the sequence $x_n$ is uniformly bounded; more precisely for every integer $n\in {\mathbb Z}$, denoting $\delta_n = x_{n+1} - x_n$, we have 
\begin{equation} \label{eq: distortion} K ^{-1} \cdot \delta_{n+1}  \leq  \delta_n \leq K \cdot  \delta_{n+1}. \end{equation}
To see this, write $x_{n+1} = g_{n} (x_n)$, where $g_n \in \mathcal G$. By definition $g_n(x_{n+1} ) \leq x_{n+2}$,  and because $g_n $ is a $K$-bilipschitz map, we get 
\[  x_{n+2} - x_{n+1} \geq g_n (x_{n+1} ) - g_n(x_n) \geq K^{-1}\cdot  (x_{n+1} - x_n),\]
hence the right inequality in (\ref{eq: distortion}). The left one is obtained by analogous considerations. This implies that $\varphi$ is close to be affine on $[x_{n-1}, x_{n+2}]$; more precisely, for every pair of points $w,z\in [x_{n-1}, x_{n+2}]$, we have 
\[ \frac{|z-w|}{K\cdot \delta_n} \leq | \varphi (z) -\varphi (w) | \leq \frac{K\cdot |z-w|}{\delta_n}.\]  

We are now able to prove that for every $g\in \mathcal G$, the map $\rho(g)$ is Lipschitz and $K (\rho(g) ) \leq K^3$. It suffices to prove that $\rho(g)$ is lipschitz on every interval of the form $[n,n+1]$ with Lipschitz constant $K^3$. Consider two points $x,y\in [n,n+1]$ and define $w= \varphi^{-1} (x)$, $z= \varphi^{-1} (y)$: we have 
\[ |\rho(g) (y) - \rho(g) (x) | \leq |\varphi ( g (z) ) - \varphi (g(w)) | \leq \frac{K\cdot |g(z)-g(w)|}{\delta_n}\leq \frac{K^2\cdot |z-w|}{\delta_n} \leq K^3 |y-x|. \]

By construction, for every element $g\in\mathcal G$,  
\[ x-2 \leq  \rho (g) (x) - x \leq x +2,  \]
because the integer points next after and before $x$ are moved a distance less than $1$ by $\rho(g)$.
Moreover, for every $n\in {\mathbb Z}$, we have $\rho(g_{n+1} g_n) (n) = n+2$. Hence, for every $x\in {\mathbb R}$ we have $\rho(g_{n+1} g_n) (x) \geq x+1$, $n$ being the integer part of $x$. Hence, we have proved that $\rho $ belongs to $R (G,\overline{\mathcal G}, K^6 , 1, 4 )$.  \end{proof}

\section{Proof of Witte's theorem}

Let $G$ be a finitely generated left-orderable group. It is a well-known fact that there is a faithful action of $G$ on the real line. By Theorem~\ref{t: almost-periodic action}, this action is topologically conjugated to an almost-periodic action, without almost fixed point. We denote by $\rho_0\in APR(G)$ the induced representation. Let $X(\rho_0)$ be the compact set constructed in Proposition~\ref{p:closure under translation flow}, the closure of $\rho_0$ under the translation flow $\Phi$. Adding to $G$ some finitely generated almost-periodic group if necessary, we can suppose that the translation flow acts freely on $X$. Now, $\rho_0$ having no almost fixed point, the group $G$ does not have any fixed point on $X(\rho_0)$.  

Suppose that $G$ is amenable. Then, there exists a probability measure $m$ on $X$ which is invariant by the representation $\mathrm{Univ}$. We are interested in the conditional measures of $m$ along the orbits of the translation flow. These are Radon measures on $m$-almost every $\Phi$-orbit $\lambda$, well-defined up to multiplication by a positive constant, and are denoted by $m_{\lambda}$. More precisely, in a flow box $[0, 1]\times \Lambda$ where the flow $\Phi$ is given by the formula $\Phi_s (t,\lambda ) = (t+s, \lambda)$, we desintegrate the measure $m$ as 
\[ m (dt , d\lambda ) = m_{\lambda} (dt )\ \overline{m} (d\lambda) \]
where $\overline{m}$ is the image of $m$ under the projection $[0, 1] \times \Lambda \rightarrow \Lambda$ and the measures $m_{\lambda}$ are measures on the unit interval. The measures $m_{\lambda}$ depend non trivially on the flow box; however, they are well-defined up to a positive constant.

Because $G$ preserves $m$, and is countable, for $m$-almost every $\Phi$-orbit $\lambda$ of $X$ and every element $g$ of $G$, the measure $g_* m_{\lambda}$ is a positive constant times $m_{\lambda}$~: 
\[  g_* m_{\lambda} = c_{\lambda} (g) m_{\lambda}\ \ \ \mathrm{where}\ \ \ c_{\lambda}(g) >0 .\] 
If $m_{\lambda}$ is not preserved by $G$, the map $g\mapsto \log c_{\lambda}$ is a non trivial morphism from $G$ to ${\mathbb R}$, and hence we deduce the existence of a non trivial morphism to $\mathbb Z$. If not, $m_{\lambda}$ is preserved by $G$, and either $m_{\lambda}$ is atomic or not. In the first case, $m_{\lambda}$ has an atom whose orbit is discrete, and the group acts as a translation on it, giving rise to a non trivial morphism to the integers. In the second case, the action is semi-conjugated to an action by translations, which defines a non trivial morphism to the reals, and hence to the integers. In all cases, $G$ carries a non trivial morphism to ${\mathbb Z}$.   

\vspace{0.54cm}

\begin{small}

\noindent{\bf Acknowledgments.} It is a pleasure to thank Victor Kleptsyn and Andr\`es Navas for stimulating conversations on group actions on the real line, and Marie-Claude Arnaud, who invented a nice argument that simplifies the proof of Lemma~\ref{l: lipschitz conjugation}.

\vspace{0.1cm}

\vspace{0.3cm}

\end{small}

\end{document}